\documentclass[leqno, 12pt, dvipsnames]{amsart}

\allowdisplaybreaks

\usepackage{palatino}
\usepackage[mathcal]{euler}

\usepackage{amssymb,amsmath,amsthm,enumerate,slashed}
\usepackage[latin1]{inputenc}
\usepackage[T1]{fontenc}
\usepackage{mathrsfs}
\usepackage{xcolor}

\definecolor{DarkBlue}{rgb}{0.2,0.2,0.4}
\usepackage[colorlinks=true, linkcolor=DarkBlue, urlcolor= DarkBlue, citecolor=DarkBlue]{hyperref}

\usepackage{graphicx}
\usepackage{epstopdf,epsfig}

\usepackage[textwidth=6in]{geometry}

\usepackage{xy}	
\xyoption{all}
\usepackage{marginnote}
\usepackage[mmddyyyy]{datetime}

\usepackage[normalem]{ulem}

\makeatletter
\def\l@subsection{\@tocline{2}{0pt}{2.5pc}{5pc}{}}
\makeatother

\newcommand{\R}{\mathbb{R}}
\newcommand{\C}{\mathbb{C}}

\newcommand{\GG}{\mathbb{G}}

\DeclareMathOperator{\Ad}{Ad}
\DeclareMathOperator{\ad}{ad}
\DeclareMathOperator{\rank}{rank}
\DeclareMathOperator{\Id}{Id}

\DeclareMathOperator{\sign}{sign}

\swapnumbers
\numberwithin{equation}{subsection}

\theoremstyle{plain}
\newtheorem*{theorem*}{Theorem}
\newtheorem*{lemma*}{Lemma}

\newtheorem{theorem}[equation]{Theorem}
\newtheorem{lemma}[equation]{Lemma}
\newtheorem{proposition}[equation]{Proposition}

\theoremstyle{definition}
\newtheorem{definition}[equation]{Definition}
\newtheorem*{definition*}{Definition}

\newtheorem*{notation*}{Notation}
\newtheorem{remark}[equation]{Remark}
\newtheorem*{remark*}{Remark}

\title[Higher Orbital Integrals on Motion Groups]{Higher orbital integrals on motion groups and Mackey deformation}
\author{Angel Rom\'an}
\address{\parbox{\textwidth}{\raggedright Department of Mathematics, Washington University, St. Louis, MO, 63130. angelr@wustl.edu}}
\author{Yanli Song}
\address{\parbox{\textwidth}{\raggedright Department of Mathematics, Washington University, St. Louis, MO, 63130. yanlisong@wustl.edu}}
\author{Xiang Tang}
\address{\parbox{\textwidth}{\raggedright
Department of Mathematics, Washington University, St. Louis, MO, 63130. xtang@wustl.edu}}
\date{\today}

\begin{document}

\begin{abstract}
We present an explicit construction of cyclic cocycles on Cartan motion groups, which can be viewed as generalizations of orbital integrals. We show that the higher orbital integral on a real reductive group associated with a semisimple element converges to the corresponding one on the associated Cartan motion group. 
\end{abstract}

\maketitle

\tableofcontents

\section{Introduction}
Mackey \cite{Mackey76} suggested that one might study unitary representations of a real reductive Lie group through representations of the associated Cartan motion group. This Mackey analogy between real reductive groups and the associated Cartan motion groups is deeply connected with the Connes-Kasparov conjecture, \cite{BCH, Higson_Mackey}. Recently notable success in the study of the Mackey analogy has led to a new proof of the Connes-Kasparov isomorphism theorem, \cite{AfgoustidisConnesKasparov}.

Let $G$ be a connected real reductive linear Lie group, and $K$ a maximal compact subgroup. Suppose that $\mathfrak{g}$ and $\mathfrak{k}$ are the corresponding Lie algebras. Consider the Cartan involution $\theta$ on $\mathfrak{g}$ such that $\mathfrak{k}$ is the eigenspace of $\theta$ associated with the eigenvalue 1. Let $\mathfrak{p}$ be the corresponding eigenspace associated with eigenvalue $-1$. The adjoint representation of $K$ on $\mathfrak{g}$ keeps $\mathfrak{k}$ invariant and therefore acts on $\mathfrak{p}$.  View $\mathfrak{p}$ as a vector group. The Cartan motion group $G_0$ is defined to be the semidirect $K\ltimes \mathfrak{p}$. When $G$ is $SL(3, \mathbb{R})$, $K$ is $SO(3)$ and $\mathfrak{p}$ can be identified with $\mathbb{R}^3$ which is equipped with the canonical $SO(3)$ action on $\mathbb{R}^3$. The Cartan motion group $G_0$ in this case is the rigid motion group of the Euclidean space $\mathbb{R}^3$. 

As a semidirect, the group $G_0$ is well studied in literature, c.f. \cite{Mackey76}. For example, the group $C^*$-algebra $C^*_r(G_0)$ is isomorphic to the crossed product algebra $C_0(\mathfrak{p}^*)\rtimes K$, whose structure is much simpler than the (reduced) group $C^*$-algebra $C^*_r(G)$ of the reductive group $G$, \cite{CCH2, CCH1, HigsonRoman20}. Connes' tangent groupoid construction \cite{connes}, which is also called deformation to normal cone, introduces a groupoid $\mathbb{G}$ a smooth family $\{G_t\}_{t\geq 0}$ of Lie groups. When $t=0$, $G_0$ is the Cartan motion group; when $t>0$, $G_t$ is isomorphic to $G$. The groupoid $\mathbb{G}$ offers a concrete and geometric link between $G_0$ and $G$, and realizes the Mackey analogy elegantly, c.f. \cite{ AfgoustidisConnesKasparov, AfgoustidisMackeyBijection, Higson_Mackey}. 

Inspired by its success in the study of $K$-theory, the last two authors of this article \cite{song2022cartan} started the exploration of the cyclic theory of $C_c^\infty(G)$ using $\mathbb{G}$. Traces on an algebra are degree 0 cyclic cocycles. It is shown \cite[Theorem 3.3]{song2022cartan} that for $x\in K$, traces on $C^\infty_c(\mathbb{G})$ defined by the orbital integral  associated with $x$ in $G_t$ converges to the corresponding orbital integral on $G_0$. This is an encouraging success to use Cartan motion group and Connes' tangent groupoid to study the cyclic theory of a reductive group. However, orbital integrals are not always sufficient to detect all $K$-theory information of $C^*_r(G)$. In \cite{Hochs-Wang}, it is observed that when $G$ does not have equal rank, orbital integrals  vanish on $K(C^*_r(G))$. To solve this defect, the last two authors \cite{ST2019higher} introduced higher degree cyclic cocycles, called higher orbital integrals, on the Harish-Chandra Schwartz subalgebra $\mathcal{C}(G_t)$ that generalize the classical orbital integrals. We recall the definition of the higher orbital integral below. 

Let $P=MAN$ the Iwasawa decomposition of a (cuspidal) parabolic subgroup. Then for any element $g\in G$, we can write $g=\kappa(g)\mu(g)e^{H(g)}\mathbf{n}(g)$ under the decomposition $G = KMAN$, where $H$ is a map from $G$ to $\mathfrak{a}$ (see \cite[Section~VII]{KnappRepTheorySemisimpleGroups}, \cite[Section~VIII]{KnappBeyond}). 
\begin{definition}(\cite[Definition~3.3]{ST2019higher})\label{dfn:highorbint}
    For any $f_0,f_1,\ldots,f_n\in \mathcal{C}(G)$ and for any regular element $x\in M$, the \emph{higher orbital integral} is defined as 
    \begin{align}
\label{definition-higher-orbital-regular}        \Phi_{P,x}(&f_0,f_1,\ldots,f_n):=  \int_{h\in M/Z_M(x)} \int_{KN} \int_{G^n} f_0(khxh^{-1}nk^{-1}g_n^{-1}\cdots g_1^{-1})\\
   \notag &f_1(g_1)\cdots f_n(g_n) \Biggl[\sum_{\sigma\in S_n} \sign(\sigma) H_{\sigma(1)}(g_1\cdots g_n k)\cdots H_{\sigma(n)}(g_n k)\Biggr]\\
   \notag & \qquad dg_1\cdots dg_n\; dk\;dn\;dh.
    \end{align}
\end{definition}

In this article, we study the property of $\Phi_{P,x}$ under the family $G_t$. For every $t>0$, we consider the cyclic cocycle $\Phi_{P_t,x,t}$ defined by a regular element $x\in T$ in $M_t$. We obtain the following main theorem. 

\begin{theorem}[Theorem \ref{thm:limit}] Let $x\in T$ be a regular\footnote{Actually we will work with a slightly stronger assumption. We will consider regular elements such that the centralizer coincides with a Cartan subgroup. Such regular elements are quite common. It is almost but not quite true that the centralizer of a regular element is Cartan as is remarked by Knapp \cite[Section~VII.8]{KnappBeyond}.} element in $M_t$ for all $t\neq 0$. The following limit holds 
    $$
    \lim_{t\rightarrow 0} \Phi_{P_t,x,t}(f_{0,t},f_{1,t},\ldots,f_{n,t})=\det\nolimits_{\mathfrak{p}}(\Ad_x-\Id)^{-1} J_\phi^{-1}\tau_{\mathfrak{a},x}(f_{0,0},f_{1,0},\ldots,f_{n,0}),
    $$
where $J_\phi = 2^{-\frac{\dim N}{2}}$ is the Jacobian of the map $\phi: \mathfrak{g}\to \mathfrak{p}$ introduced in (\ref{eq:phi}).
\end{theorem}

In the above theorem, $\tau_{\mathfrak{a},x}$ is an analog of the $\Phi_{P,x}$ on the Cartan motion group. Let $A\subset \mathfrak{p}$ be a subspace with $\dim A=n$. Naturally, we can write $\mathfrak{p}=A\oplus A^\perp$. In the case of a Cartan motion group $G_0=K\ltimes \mathfrak{p}$, the group $M_K$ is just $Z_K(\mathfrak{a})$, where $\mathfrak{a}$ is a subspace of a maximal abelian Lie subalgebra $\mathfrak{a}_s$ of $\mathfrak{p}$. Note then that $Z_{G_0}(\mathfrak{a})=M_K\times\mathfrak{a}$. Then we have an ``Iwasawa" decomposition of the Cartan motion group
$$
G_0=(KM_K)\ltimes (\mathfrak{a}\oplus\mathfrak{a}^\perp),$$
where $\mathfrak{a}^\perp$ is the orthogonal complement of $\mathfrak{a}$ in $\mathfrak{p}$. We may think of $M_K\times \mathfrak{a}$ as analogous to the group $M$ in the Iwasawa decomposition of the reductive group $G=KMAN$. Generalizing Definition \ref{dfn:highorbint}, we obtain the following cyclic cocycle on $G_0$. 

\begin{definition}(Definition \ref{regular-general-rank-definition})
Let $f_0,f_1,\ldots,f_n\in C^\infty_c(G_0)$. For any $\sigma \in S_n$, we define 
\begin{align}
\notag &\tau_{A,x,\sigma}(f_0, f_1, \dots, f_n)\\
\label{regular-higher-orbital-general-rank} \colon =&\int_{h\in M_K/Z_{M_K}(x)} \int_{u \in K} \int_{w \in A^\perp}\int_{G_0^n} \\
& \notag f_0\left(uhxh^{-1}u^{-1}k_n^{-1}\cdots k_1^{-1},k_1\ldots k_n uw -\sum_{j=1}^n \left(\prod_{l=1}^{j} k_l\right)v_j\right) \\
\notag &\prod_{j=1}^n f_j(k_j, v_j) \cdot  \prod_{j=1}^n H_{\sigma(j)}\left(u^{-1} \left(\prod^n_{l=j+1} k_l\right)^{-1} v_j\right),
\end{align}
thus, we define the \emph{higher orbital integral with respect to $x\in M$} as
\begin{equation}
\tau_{A,x}(f_0,f_1,\ldots,f_n)= \sum_{\sigma\in S_n} \sign(\sigma) \tau_{A,x,\sigma}(f_0,f_1,\ldots,f_n).
\end{equation}
\end{definition}

By the Fourier transform, $\mathcal{C}(G_0)$ is isomorphic to $\mathcal{C}(\mathfrak{p}^*)\rtimes K$. As the index pairing between the $K$-theory and cyclic cohomology on $G_t$ is continuous with respect to $t$, Theorem \ref{thm:limit} and the cyclic cocycle $\tau_{\mathfrak{a},x}$ on $\mathcal{C}(\mathfrak{p}^*)\rtimes K$ suggest that we might be able to compute the pairings on $G_t$ for $t>0$ by the pairing on $G_0$, whose geometry information is easier to get hold of, and obtain a better understanding of \cite[Theorem II]{ST2019higher}. We plan to study this question in the near future. 

We would like to remark that the cyclic (co)homology of the crossed product algebra $\mathcal{C}(\mathfrak{p}^*)\rtimes K$ is well studied in literature \cite{Brylinski:Grenoble, Nistor, PflaumPosthumaTang:20}. The main contribution of this article is the construction of explicit formulas of cyclic cocycles. A better understanding of these new cocycles will lead to interesting new formulas in equivariant index theory. 

Our paper is organized as follows. In Section \ref{sec:prelim}, we review the background information needed to read this article. In particular, we include a brief introduction to cyclic cohomology, higher orbital integrals on a reductive group, the groupoid $\mathbb{G}$ together with some integral formulas. The introduction of the higher orbital integrals on a Cartan motion group is presented  in Section \ref{section-higher-orbital-integral-motion-group}. And we prove the main theorem about the limit of higher orbital integrals in Section \ref{section-deformation-higher-orbital-integral}. \\

\noindent{\bf Acknowledgments}: We would like to thank 
Alexander Afgoustidis, Pierre Clare, Axel Gastaldi, Satwata Hans, Nigel Higson, and Hang Wang for inspiring discussions. Roman's research was partially supported by the NSF grant DMS-2213097. Song's research was partially supported by the NSF grant  DMS-1952557. Tang's research was partially supported by the NSF grants DMS-1952551, DMS-2350181, and Simons Foundation grant MPS-TSM-00007714. The authors acknowledge support of the Institut Henri Poincar\'e (UAR 839 CNRS-Sorbonne Universit\'e), and LabEx CARMIN (ANR-10-LABX-59-01).

\begin{notation*} We shall use the following notations: $\prod_{l=j}^n k_l=k_j\cdots k_n$. Furthermore, we shall take the convention that $\prod_{l=n+1}^n k_l=e$, the identity element.
\end{notation*}

\section{Preliminaries}\label{sec:prelim}

In this section, we briefly introduce the background information about cyclic cohomology and orbital integrals. 
\subsection{Cyclic Cohomology}

We will define the cyclic cohomology for an algebra.

Let $A$ be a Fr\'echet algebra over $\C$. The \emph{space $C^n(A)$ of Hoschild cochains of degree $n$} is the space of all bounded complex-valued $(n+1)$-linear functionals on $A$. 
The \emph{Hoschild codifferential} is the map $b: C^n(A)\rightarrow C^{n+1}$ defined by
\begin{align*}
    b\Phi\;(a_0, a_1,\cdots , a_{n+1})=&\sum_{k=0}^n (-1)^k\Phi(a_0,\cdots, a_k\cdot a_{k+1},\cdots, a_{n+1})\\
    &\qquad+ (-1)^{n+1}\Phi(a_{n+1}\cdot a_0,\cdots, a_n).
\end{align*}
The \emph{Hoschild cohomology} $HH^*(A)$ is the cohomology of the complex $(C^*(A),b)$.

We say that a cochain $\Phi\in C^n(A)$ is \emph{cyclic} if it satisfies
$$
\Phi(a_n,a_0,\ldots,a_{n-1})=(-1)^n\Phi(a_0,a_1,\cdots,a_n)$$
for all $a_0,a_1,\ldots,a_n\in A$. The subspace $C^*_\lambda(A)$ of cyclic cochains is closed under the Hoschild codifferential $b$. Thus the \emph{cyclic cohomology} $HC^*(A)$ is the cohomology of the subcomplex $(C^*_\lambda(A),b)$.

\subsection{Measures and Integral Formulas}

Throughout this paper, we will let $G$ be a (connected) real reductive group and $K$ a maximal compact subgroup of $G$. In Section~\ref{section-higher-orbital-integral-motion-group}, let $V$ be a (finite-dimensional) Euclidean space such that $K$ acts isometrically on it. The \emph{motion group} is defined as $G_0:=K\ltimes V$ with the product law 
\[
(k_0, v_0) \cdot (k_1, v_1) \colon = (k_0k_1, k_1^{-1}v_0 +  v_1).
\]
In Section~\ref{section-deformation-higher-orbital-integral}, there is a motion group associated to the reductive group $G$ called the \emph{Cartan motion group}. Let $\mathfrak{g}=\mathfrak{k}\oplus\mathfrak{p}$ be the Lie algebra of $G$ with the Cartan decomposition. Then the associated Cartan motion group is $G_0:=K\ltimes \mathfrak{p}$ with the product law
\[
(k_0, v_0) \cdot (k_1, v_1) \colon = (k_0k_1, \Ad_{k_1^{-1}}v_0 +  v_1).
\]

Let $T$ be the Cartan subgroup of $K$. We define a Haar measure $dk$ on $K$ such that $\int_T e \;dk=1$. On $V$, we let $dv$ be the standard Euclidean measure. Then the measure on $G_0$ is $dk\;dv$. That is
$$
\int_{G_0} f(g_0)\;dg_0=\int_K\int_V f(k,v)\;dk\;dv.$$

Let $A$ be a subspace of $V$ and let $A^\perp$ be the orthogonal complement of $A$ in $V$; that is, $V=A\oplus A^\perp$. Then
$$
\int_Vf(v)\;dv=\int_A \int_{A^\perp} f(a+a')\;da\;da'.$$

Finally, if $x\in T$, the orbital integral has the property that 
$$
\int_{K/T} f(kxk^{-1})\;d[k]=\int_Kf(kxk^{-1})\;d[k],
$$
as the volume of $T$ is assumed to be 1.

We recall the \emph{(global) Cartan decomposition}. 
\begin{theorem}
    The mapping $K\times \mathfrak{p}\rightarrow G$ defined by $(k,X)\mapsto k\exp(X)$ is a diffeomorphism.
\end{theorem}
We refer the reader to \cite[Theorem~6.31]{KnappBeyond} for a proof in the case of semi-simple groups\footnote{In \cite{KnappBeyond}, the global Cartan decomposition is part of the \emph{definition} of a reductive group.}. We describe the integral formula under this decomposition.
\begin{proposition}
    Up to a multiplicative constant the following holds 
    \begin{equation}\label{integral-Cartan-decompositon} \int_G f(g)\;dg=\int_K\int_\mathfrak{p} f(k\exp(X))\;J(X)\;dk\;dX
    \end{equation}
    where $J(X)$ is the absolute value of the determinant of the map $\mathfrak{p}\rightarrow \mathfrak{p}$ defined by the operator
    $$\frac{\sinh(\ad_X)}{\ad_X}.$$
\end{proposition} 
This result is well-known, but we refer the reader to \cite[Lemma 2.5.6]{Higson:lecturenotes} for a proof. From now on, we normalize the measures and choose a multiplicative constant so that \eqref{integral-Cartan-decompositon} actually holds.

The reductive group $G$ has another decomposition, the $KAK$ decomposition. We follow the convention established in \cite[Proposition~5.28]{KnappRepTheorySemisimpleGroups}) in the following. 
\begin{proposition}
    The following identity holds
    \begin{align}\label{integral-KAK-decomposition}
    \int_G f(g)\;dg=&\int_K\int_{\mathfrak{a}^+}\int_K f(k_1\exp(H)k_2)\\
   \notag &\qquad \prod_{\alpha\in \Delta^+(\mathfrak{\mathfrak{g},\mathfrak{a})}} \sinh(\alpha(H))^{\dim\mathfrak{g}_\alpha}\;dk_1\;dH\;dk_2, \end{align}
    where $\Delta^+(\mathfrak{g},\mathfrak{a})$ is the set of positive restricted roots of the $\mathfrak{a}$-action on $\mathfrak{g}$, $\mathfrak{a}^+$ is corresponding positive Weyl chamber, and $\mathfrak{g}_\alpha$ is the root space of the root $\alpha$.
\end{proposition}
See \cite[Lemma~2.4.2]{Wallach1} for a proof.

Let $H$ be a Cartan subgroup of $G$. We can normalize measures on $H$ and on $G/H$ so that 
$$
\int_G f(g)\;dg=\int_{G/H}\left[\int_H f(gh)\;dh\right]\;d(gH).$$
Let $x\in G$ be a regular element such that the centralizer $Z_G(x)$ is a Cartan subgroup (see Remark~\ref{remark-on-regular-elements}). The orbital integral $\Lambda_x(f)$, formally defined by
\begin{equation}\label{orbital-integral}
\Lambda_x(f)=\int_{G/Z_G(x)} f(gxg^{-1})\;d[g],
\end{equation}
is well-defined whenever $f\in C^\infty_c(G)$. See Appendix~A in \cite{ST2019higher} (in particular, Proposition~A.4) for the case of Harish-Chandra's Schwartz functions. We also refer the reader to \cite[Theorem~6]{HarishChandraDiscreteSeries2}.

let $P=MAN$ be a cuspidal parabolic subgroup. Then we have the Iwasawa decomposition $G=KMAN$. The integral formula under this decomposition is
\begin{equation}\label{integral-Iwasawa-decomposition}
\int_G f(g)\;dg=\int_K\int_M\int_A\int_N f(kman)e^{2\rho(\log a)}\;dk\;dm\;da\;dn,
\end{equation}
where $\rho$ is the half-sum of positive roots.

Next, we list two more integral formulas. Let $\mathfrak{a}$ and $\mathfrak{n}$ be the corresponding Lie algebras of $A$ and $N$, respectively. That is, $\exp[\mathfrak{a}]=A$ and $\exp[\mathfrak{n}]=N$. Then
\begin{align}
\label{integral-group-to-algebra-A}    \int_A f(a)\;da & = \int_{\mathfrak{a}}f(\exp(X))\;dX,\\
\label{integral-group-to-algebra-N}    \int_N f(n)\;dn & = \int_{\mathfrak{n}} f(\exp(Y))\;dY.
\end{align}

Note that $\mathfrak{g}=\mathfrak{k}\oplus\mathfrak{p}$ and $\mathfrak{g}=\mathfrak{k}\oplus\mathfrak{a}_s\oplus\mathfrak{n}_s$. The subscript $s$ in the previous decomposition signifies maximally split. We need a way to relate the Cartan decomposition and the Iwasawa decomposition. Let $\phi:\mathfrak{g}\rightarrow \mathfrak{p}$ be defined by
\begin{equation}\label{eq:phi}
\phi(X)=\frac{1}{2}(X-\theta X)
\end{equation}
where $\theta$ is the Cartan involution. We note that $\phi\vert_{\mathfrak{a}_s\oplus\mathfrak{n}_s}$ gives us an isomorphism between $\mathfrak{a}_s\oplus\mathfrak{n}_s$ and $\mathfrak{p}$. In particular, if we let $\mathfrak{a}_s^\perp$ be the orthogonal complement of $\mathfrak{a}_s$ in $\mathfrak{p}$, then we can see that $\mathfrak{a}_s^\perp$ is isomorphic to $\mathfrak{n}_s$. The Jacobian of $\phi$, which we will denote by $J_\phi$, is constant. We obtain the following integral formula
\begin{equation}\label{integral-Lie-algebra-diffeo}
\int_{\mathfrak{n}} (f\circ \phi)(Y)\;dY=\int_{\phi[\mathfrak{n}]} f(w)\;J_\phi^{-1}\;dw.
\end{equation}

\subsection{Higher Orbital Integrals for Reductive Groups}
\label{subsection-review-orbital-integral-reductive}

In this subsection, we review the definition and the main theorems of the higher orbital integral for reductive group constructed by Song and Tang in \cite{ST2019higher}. Note that in that paper, they worked with the algebra of Harish-Chandra's Schwartz function $\mathcal{C}(G)$. Since $C^\infty_c(G)\subset \mathcal{C}(G)$, thus all results apply in this paper.

Let $G=KMAN$ be a reductive group together with its Iwasawa decomposition. Then for any element $g\in G$, we can write $g=\kappa(g)\mu(g)e^{H(g)}\mathbf{n}(g)$ under the decomposition. We write $P=MAN$, the (cuspidal) parabolic subgroup.
\begin{definition}(\cite[Definition~3.3]{ST2019higher})
    For any $f_0,f_1,\ldots,f_n\in \mathcal{C}(G)$ and for any regular element $x\in M$, the \emph{higher orbital integral} is defined as
    \begin{align}
\Phi_{P,x}(f_0,f_1,\ldots,f_n):= & \int_{h\in M/Z_M(x)} \int_{KN} \int_{G^n} f_0(khxh^{-1}nk^{-1}g_n^{-1}\cdots g_1^{-1})\\
   \notag &f_1(g_1)\cdots f_n(g_n) \Biggl[\sum_{\sigma\in S_n} \sign(\sigma) H_{\sigma(1)}(g_1\cdots g_n k)\cdots H_{\sigma(n)}(g_n k)\Biggr]\\
   \notag & \qquad dg_1\cdots dg_n\; dk\;dn\;dh.
    \end{align}
\end{definition}
In Appendix~A of \cite{ST2019higher}, the authors show that the integral converges.

\begin{theorem}(\cite[Theorem~3.5]{ST2019higher})
For all regular element $x\in M$, the higher orbital integral $\Phi_{P,x}$ is a cyclic cocycle and defined an element
$$
[\Phi_{P,x}]\in HC^n(\mathcal{C}(G)).$$
\end{theorem}

 When $P$ is the maximal cuspidal parabolic subgroup of $G$ and $x$ runs through all regular elements, the cocycles $\{\Phi_{P,x}\}$ generate the (periodic) cyclic cohomology of $\mathcal{C}(G)$. We refer the readers to \cite{ST2019higher} for more detailed discussion and the computation of the pairing between $[\Phi_{P,x}]$ and $K$-theory of $C^*_r(G)$.

\subsection{Deformation to the Normal Cone}\label{subsection-deformation}

There are many references for the definition of the deformation to the normal cone in terms of smooth manifolds. See, for example, \cite{HigsonRoman20}. We will follow more closely to that in \cite{song2022cartan} and use the definition relevant to this paper.

The \emph{deformation to the normal cone} of a reductive group $G$ over a maximal compact subgroup $K$ is the set
$$
\GG:=(G_0\times \{0\})\;\sqcup (G\times \R^\times),$$
where $G_0=K\ltimes \mathfrak{p}$ is the associated Cartan motion group, together with a unique smooth structure such that the map $K\times \mathfrak{p}\times\R\rightarrow \GG$ defined by
$$
(k,X,t)\mapsto \begin{cases} (k\exp(tX),t) & t\neq 0\\
(k,X,0) & t=0\end{cases}$$
is a diffeomorphism. Thus, a function $f$ on $\GG$ is smooth if and only if there is a smooth function $\tilde{f}$ on $K\times \mathfrak{p}\times \R$ such that
$$\tilde{f}(k,X,t)=\begin{cases} f(k\exp(tX),t) & t\neq 0\\ f(k,X,0) & t=0.\end{cases}$$

Let $G_t$ be a fiber $\GG$ over $t\in \R$, that is, $G_t=\{(k\exp(tX),t) : k\in K, X\in \mathfrak{p}\}$. Note that we have $\GG=\sqcup_{t\in \R} G_t$. If we fix a Haar measure $dg$ for $G=G_1$, then, for $t\neq 0$, we define a Haar measure $d_tg$ on $G_t$ as $d_tg=|t|^{-\dim\mathfrak{p}}\;dg$. We will define the measure on $G_0$ as $dk\;dX$, as usual.

\begin{proposition}
    The family of Haar measure $d_tg$ can be given by the formula 
$$    \int_{G_t} f(g,t)\;d_tg = \int_K\int_\mathfrak{p} f(k\exp(tX),t)J(tX)\;dk\;dX. $$
\qed
\end{proposition}

\begin{lemma}\label{lemma-jacobian-limit}
    The limit of the Jacobian $J(tX)$ as $t\rightarrow 0$ is one, that is,
    $$
    \lim_{t\rightarrow 0} J(tX)=1.$$
\end{lemma}

\begin{proof}
    Note that
    $$
    \sinh(\ad_{tX})=\ad_{tX}+\frac{1}{3!}\ad_{tX}^3+\frac{1}{5!}\ad_{tX}^5+\cdots$$
    so then
    \begin{align*}
        \frac{\sinh(\ad_{tX})}{\ad_{tX}}=\Id+\frac{1}{3!}\ad_{tX}^2+\frac{1}{5!}\ad_{tX}^4+\cdots.
    \end{align*}
    Then we obtain
    $$
    \lim_{t\rightarrow 0} \frac{\sinh(\ad_{tX})}{\ad_{tX}}=\Id.$$
    Since the determinant is continuous, we get the statement of the lemma.
\end{proof}

\begin{remark*} In \cite{HigsonRoman20}, the authors characterized the smooth structure on $\GG$ using the Iwasawa decomposition. First, the fiber $G_0$ was defined as $K\ltimes (\mathfrak{g}/\mathfrak{k})$. Then a diffeomorphism $K\times \mathfrak{a}\times\mathfrak{n}\times \R\rightarrow \GG$ was constructed with the formula
$$
(k,X,Y,t)\mapsto \begin{cases}
    (k\exp(tX)\exp(tY),t) & t\neq 0 \\
    (k,[X+Y],0) & t=0,
\end{cases}
$$
where $[X+Y]$ is the class in the quotient space $\mathfrak{g}/\mathfrak{k}$. We will modify this to better serve our purpose. 
\end{remark*}

\begin{lemma}
    The following limit holds under the smooth structure of $\GG$
    $$\lim_{t\rightarrow 0} (k\exp(tX)\exp(tY),t) = (k,X+\phi(Y),0).
$$
\end{lemma}

\begin{proof}
    We have
    $$
    (k\exp(tX)\exp(tY),t)=(k\exp(t(X+Y)+O(t^2),t).$$\
    Let $Z=X+Y$, an element of $\mathfrak{g}$. Let
    \begin{align*}
        Z_1=&\frac{1}{2}(X+\theta X) \in \mathfrak{k}\\
        Z_2=&\phi(X)=\frac{1}{2}(X-\theta X) \in \mathfrak{p},
    \end{align*}
so that $Z=Z_1+Z_2$. Then
\begin{align*}
    (k\exp(t(Z_1+Z_2)+O(t^2)),t) & = (k\exp(tZ_1)\exp(tZ_2+O(t^2)),t)\\
    & \rightarrow (k,Z_2,0)
\end{align*}
as $t\rightarrow 0$. But
$$Z_2=\phi(Z)=\phi(X+Y)=X+\phi(Y).$$
    
\end{proof}

\

We recall one more important theorem, this time from \cite{song2022cartan}.
\begin{proposition}(\cite[Definition~3.2]{song2022cartan})
    For every positive real number $t$, the Haar measure $d_tg$ on $G_t$ is given by the following formula 
\begin{align*}    \int_{G_t} f(g,t)\;d_tg = t^{\dim\mathfrak{a}_s-\dim\mathfrak{p}} &\int_K\int_{\mathfrak{a}^+}\int_K f(k_1\exp(tH)k_2,t)\\
& \qquad \prod_{\alpha\in \Delta^+(\mathfrak{g},\mathfrak{a})} \sinh(\alpha(tH))^{\dim \mathfrak{g}_\alpha}\;dk_1\;dH\;dk_2.
\end{align*}\qed
\end{proposition}

\begin{theorem}(\cite[Theorem~3.3]{song2022cartan})
    Suppose that $G$ is a reductive group and $K$ is its maximal compact subgroup. We assume that $\rank G=\rank K$. If $f\in C^\infty_c(\GG)$ and if $x\in K$ is a regular element, then
    \begin{equation*}
        \lim_{t\rightarrow 0} \int_{G_t} f(gxg^{-1},t)\;d_tg = \int_{K} \int_\mathfrak{p} f(kxk^{-1},\Ad_{kxk^{-1}} X-X,0)\;dk\;dX.
    \end{equation*}
\end{theorem}

\section{Definition of higher orbital integrals for motion groups}\label{section-higher-orbital-integral-motion-group}

In this section, we introduce higher orbital integrals on motion groups generalizing the construction on reductive groups in \cite{ST2019higher}. We then prove that it is a cyclic cocycle.

\subsection{Definition of higher orbital integral}

Let $G_0=K\ltimes V$ be the motion group, where $K$ is a compact group acting isometrically on a (finite-dimensional) Euclidean space $V$, with the product law
$$
(k_0,v_0)\cdot(k_1,v_1)=(k_0k_1,k_1^{-1}v_0+v_1).$$
Here, whenever $k\in K$ and $v\in V$, the element $kv\in V$ denotes the image of $v$ under the action of $k$. Let $A\subset V$ be a subspace with $\dim A=n$. Naturally, we can write $V=A\oplus A^\perp$. 
\begin{definition} \label{definition-H-components} Fix an orthonormal basis $e_1,\ldots,e_n$ for $A$ and define a linear map $H_j:V\rightarrow \R$ by 
$$
H_j(v)=\langle v,e_j\rangle .$$
\end{definition}

Let 
\begin{equation}\label{stabilizer} M_K:=\{k\in K : kv=v,\; v\in A\}\end{equation} be the stabilizer of $A$ in $K$. 

\begin{remark}\label{M-K-Remark} In the case of a Cartan motion group $G_0=K\ltimes \mathfrak{p}$, the group $M_K$ is just $Z_K(\mathfrak{a})$, where $\mathfrak{a}$ is a subspace of a maximal abelian Lie subalgebra $\mathfrak{a}_s$ of $\mathfrak{p}$. Note then that $Z_{G_0}(\mathfrak{a})=M_K\ltimes\mathfrak{a}$. Then we have an ``Iwasawa" decomposition of the Cartan motion group
$$
G_0=(KM_K)\ltimes (\mathfrak{a}\oplus\mathfrak{a}^\perp),$$
where $\mathfrak{a}^\perp$ is the orthogonal complement of $\mathfrak{a}$ in $\mathfrak{p}$. We may think of $M_K\ltimes \mathfrak{a}$ as analogous to the group $M$ in the Iwasawa decomposition of the reductive group $G=KMAN$. 
\end{remark}

\begin{lemma}\label{lemma-orthogonal-complement-is-invariant}
    The orthogonal complement $A^\perp$ of $A$ in $V$ is $M_K$-invariant.
\end{lemma}

\begin{proof}
    Recall that the action of $K$ on $V$ is isometrically. So let $w\in A^\perp$ and let $k\in K$. For any $v\in A$, we have 
    $$\langle kw,v\rangle=\langle w,k^{-1}v\rangle=\langle w,v\rangle = 0.$$
    Thus, $kw\in A^\perp$.
\end{proof}

Let $x\in M_K$. Recall that $S_n$ is the permutation group of $n$ letters.
\begin{definition}\label{regular-general-rank-definition}
Let $f_0,f_1,\ldots,f_n\in C^\infty_c(G_0)$. For any $\sigma \in S_n$, we define 
\begin{align}
\notag &\tau_{A,x,\sigma}(f_0, f_1, \dots, f_n)\\
\colon =&\int_{h\in M_K/Z_{M_K}(x)} \int_{u \in K} \int_{w \in A^\perp}\int_{G_0^n} \\
& \notag f_0\left(uhxh^{-1}u^{-1}k_n^{-1}\cdots k_1^{-1},k_1\ldots k_n uw -\sum_{j=1}^n \left(\prod_{l=1}^{j} k_l\right)v_j\right) \\
\notag &\prod_{j=1}^n f_j(k_j, v_j) \cdot  \prod_{j=1}^n H_{\sigma(j)}\left(u^{-1} \left(\prod^n_{l=j+1} k_l\right)^{-1} v_j\right),
\end{align}
thus, we define the \emph{higher orbital integral with respect to $x\in M_K$} as
\begin{equation}
\label{regular-higher-orbital-integral-sum} 
\tau_{A,x}(f_0,f_1,\ldots,f_n)= \sum_{\sigma\in S_n} \sign(\sigma) \tau_{A,x,\sigma}(f_0,f_1,\ldots,f_n). 
\end{equation}
\end{definition}
In Definition \ref{regular-general-rank-definition}, we have omitted the measure ${\rm{d}}h{\rm{d}}u{\rm d}w{\rm d}k_1\cdots {\rm d}k_n$. In the remaining part of the paper, we will omit the measure when it is clear. 
\begin{theorem}
    For any element $x\in M_K$, the higher orbital integral defined in \eqref{regular-higher-orbital-integral-sum} is a cyclic cocycle.
\end{theorem}
The proof will be given in the next two subsections.

\subsection{Cocycle Condition}

    To prove that it is a cocycle, we wish to show
    \begin{align}
        \label{regular-cocycle}\tau_{A,x}(f_0*f_1,f_2,\ldots,f_{n+1})+ \sum_{i=1}^n (-1)^i \tau_{A,x}(f_0,\ldots,f_i*f_{i+1},\ldots,f_{n+1})\\\notag +(-1)^{n+1}\tau_{A,x}(f_{n+1},f_0,\ldots,f_{n})=0.
    \end{align}
    We have 
    \begin{align*}
         \tau_{A,x}(&f_0*f_1,f_2,\ldots,f_{n+1}) =\int_{h\in M_K/Z_{M_K}(x)}\int_{u\in K} \int_{w\in A^\perp}\int_{G_0^n}\int_{(k,v)\in G_0} f_0(k,v)\\
        & f_1\Bigl(k^{-1}uhxh^{-1}u^{-1}k_n^{-1}\cdots k_1^{-1},-k_1\cdots k_nuht^{-1}h^{-1}u^{-1} kv+k_1\cdots k_nuw\\
        & -\sum_{j=1}^n\left(\prod_{l=1}^j k_l\right)v_j\Bigr)\\
        & \prod_{j=1}^nf_{j+1}(k_j,v_j)\left[\sum_{\sigma\in S_n}\sign(\sigma)\prod_{j=1}^nH_{\sigma(j)}\left(u^{-1}\left(\prod_{l=j+1}^{n}k_l\right)^{-1}v_j\right)\right].
    \end{align*}
We make the following change of variables
\begin{align*}
    k'_1=k^{-1}uhxh^{-1}u^{-1}k_n^{-1}\cdots k_1^{-1}  \Rightarrow k=uhxh^{-1}u^{-1}k_n^{-1}\cdots k_1^{-1}k'^{-1}_1\\
    k_{j+1}'=k_j\qquad (j=1,\ldots,n)
\end{align*}
to obtain
\begin{align*}
     \int_{h\in M_K/Z_{M_K}(x)}\int_{u\in K}&\int_{w\in A^\perp}\int_{G_0^{n+1}} f_0(uhxh^{-1}u^{-1}k'^{-1}_{n+1}\cdots k'^{-1}_2 k'^{-1}_1,v) \\
    & f_1\left(k'_1,-k'^{-1}_1v+k'_2\cdots k'_{n+1}uw- \sum_{j=1}^n\left( \prod_{l=1}^j k'_{l+1} \right) v_j \right)\\
    & \prod_{j=1}^n f_{j+1}(k'_{j+1},v_j)\left[\sum_{\sigma\in S_n} \sign(\sigma) \prod_{j=1}^n H_{\sigma(j)}\left(u^{-1}\left( \prod_{l=j+1}^n k'_{l+1}  \right)^{-1} v_j \right)   \right].
\end{align*}
We make another change of variables
\begin{align*}
    & v'_1=-k'^{-1}_1v+k'_2\cdots k'_{n+1}uw-\sum_{j=1}^n\left(\prod_{l=1}^j k'_{l+1}\right)v_j\\
    \Rightarrow & v=-k'_1v'_1+k'_1\cdots k'_{n+1} uw -\sum_{j=1}^n\left(\prod_{l=1}^{j+1} k'_l\right)v_j\\
    & v'_{j+1}=v_j\qquad (j=1,\ldots,n)
\end{align*}
to obtain
\begin{align}\label{regular-integral-A}
\tau_{A,x}(&f_0*f_1,f_2,\ldots,f_{n+1})\\
    \notag \int_{h\in M_K/Z_{M_K}(x)} & \int_{u\in K}\int_{w\in A^\perp}\int_{G_0^{n+1}}\\
   \notag &f_0\left(uhxh^{-1}u^{-1}k'^{-1}_{n+1}\cdots k'^{-1}_2 k'^{-1}_1,k'_1\cdots k'_{n+1}uw-\sum_{j=1}^{n+1}\left(\prod_{l=1}^{j}k'_l\right)v'_{j}\right) \\
    \notag & \prod_{j=1}^{n+1} f_j(k'_j,v'_j)\left[\sum_{\sigma\in S_n} \sign(\sigma) \prod_{j=1}^n H_{\sigma(j)}\left(u^{-1}\left( \prod_{l=j+1}^n k'_{l+1}  \right)^{-1} v'_{j+1} \right)   \right].
\end{align}

Now we look at the following terms. Let $i=1,\ldots,n$. We have 
\begin{align*}
    \tau_{A,x}(&f_0,\ldots,f_i*f_{i+1},\ldots,f_{n+1}) = \\
    \int_{h\in M_K/Z_{M_K}(x)}\int_{u\in K}& \int_{w\in A^\perp} \int_{G_0^n} \int_{(k,v)\in G_0} \\
   & f_0\left(uhxh^{-1}u^{-1}k_n^{-1}\cdots k_1^{-1}, k_1\ldots k_n uw -\sum_{j=1}^n \left( \prod_{l=1}^j k_l  \right)v_j\right)\\
  & \prod_{j=1}^{i-1} f_j(k_j,v_j) f_i(k,v)f_{i+1}(k^{-1}k_i,-k_i^{-1}kv+v_i)\prod_{j=i+1}^n f_{j+1}(k_j,v_j)\\
  & \left[ \sum_{\sigma\in S_n} \sign(\sigma) \prod_{j=1}^n H_{\sigma(j)} \left(u^{-1}\left( \prod_{l=j+1}^n k_l \right)^{-1} v_j   \right)          \right].
\end{align*}
We make the first change of variables as follows.
\begin{align*}
    & k'_i=k\quad\& \quad  k'_{i+1}=k^{-1}k_i \qquad \Rightarrow \qquad k_i=k'_i k'_{i+1} \\
    & k'_j=k_j \qquad (j=1,\ldots,i-1)\\
    & k'_{j+1}=k_j \qquad (j=i+1,\ldots,n)
\end{align*}
to obtain
\begin{align*}
    \int_{h\in M_K/Z_{M_K}(x)}&\int_{u\in K}  \int_{w\in A^\perp} \int_{G_0^{n+1}} f_0(uhxh^{-1}u^{-1}k^{-1}_{n+1}\cdots k'^{-1}_1,k'_1\cdots k'_{n+1}uw\\
    &  \qquad -k'_1v_1-\cdots-k'_1\cdots k'_{i-1}v_{i-1}-k'_1\cdots k'_ik'_{i+1}v_i-\cdots k'_1\cdots k'_{n+1}v_n)\\
    & \prod_{j=1}^{i-1} f_j(k'_j,v_j) f_i(k'_i,v)f_{i+1}(k'_{i+1},-k'^{-1}_{i+1}v+v_i)\prod_{j=i+1}^n f_{j+1}(k'_{j+1},v_j) \\
    &\Bigl[ \sum_{\sigma\in S_n}\sign(\sigma) H_{\sigma(1)}(u^{-1}k'^{-1}_{n+1}\cdots k'^{-1}_2 v_1)\cdots H_{\sigma(i-1)}(u^{-1}k'_{n+1}\cdots k'^{-1}_i v_{i-1})\\
    &\qquad H_{\sigma(i)}(u^{-1}k'^{-1}_{n+1}\cdots k'^{-1}_{i+2} v_i)\cdots H_{\sigma(n-1)}(u^{-1}k'^{-1}_{n+1} v_{n-1})H_{\sigma(n)}(u^{-1}v_n)               \Bigr].
\end{align*}
Now we make the next change of variables as follows
\begin{align*}
    & v'_i=v\quad \&\quad v'_{i+1}=-k'^{-1}_{i+1}v+v_i \qquad \Rightarrow \qquad v_i=k'^{-1}_{i+1}v'_i+v'_{i+1}\\
    & v'_j=v_j\qquad (j=1,\ldots,i-1)\\
    & v'_{j+1}=v_j\qquad (j=i+1,\ldots,n)
\end{align*}
to obtain
\begin{align}
    \notag &\int_{h\in M_K/Z_{M_K}(x)}\int_{u\in K}  \int_{w\in A^\perp} \int_{G_0^{n+1}} f_0(uhxh^{-1}u^{-1}k'^{-1}_{n+1}\cdots k'^{-1}_1,k'_1\cdots k'_{n+1}uw\\
   \notag &  -k'_1v'_1-\cdots-k'_1\cdots k'_{i-1}v'_{i-1}-k'_1\cdots k'_ik'_{i+1}(k'^{-1}_{i+1}v'_i+v'_{i+1})-\cdots k'_1\cdots k'_{n+1}v'_{n+1})\\
  \notag  & \prod_{j=1}^{i-1} f_j(k'_j,v'_j) f_i(k'_i,v'_i)f_{i+1}(k'_{i+1},v'_{i+1})\prod_{j=i+1}^n f_{j+1}(k'_{j+1},v'_{j+1}) \\
  \notag  &\Bigl[ \sum_{\sigma\in S_n}\sign(\sigma) H_{\sigma(1)}(u^{-1}k'^{-1}_{n+1}\cdots k'^{-1}_2 v'_1)\cdots H_{\sigma(i-1)}(u^{-1}k'_{n+1}\cdots k'^{-1}_i v'_{i-1})\\
  \notag  &\qquad H_{\sigma(i)}\bigl(u^{-1}k'^{-1}_{n+1}\cdots k'^{-1}_{i+2}(k'^{-1}_{i+1} v'_i+v'_{i+1}\bigr))\cdots H_{\sigma(n-1)}(u^{-1}k'^{-1}_{n+1} v'_n)H_{\sigma(n)}(u^{-1}v'_{n+1})               \Bigr]\\
  \label{regular-integral-B} =&\int_{h\in  M_K/Z_{M_K}(x)} \int_{u\in K} \int_{w\in A^\perp} \int_{G_0^{n+1}}\\
 \notag & f_0\left(uhxh^{-1}u^{-1}k'^{-1}_{n+1}\cdots k'^{-1}_1, k'_1\cdots k'_{n+1}uw - \sum_{j=1}^{n+1} \left(\prod_{l=1}^j k'_l\right) v'_j\right) \prod_{j=1}^{n+1} f_j(k'_j,v'_j)\\
 \notag & \biggl[\sum_{\sigma\in S_n} \sign(\sigma)\prod_{j=1}^{i-1} H_{\sigma(j)}\Bigl(u^{-1}\bigl(\prod_{l=j+1}^{n+1} k'_l\bigr)^{-1}v'_j\Bigr)  \\
 \notag & \qquad H_{\sigma(i)}\left(u^{-1}k'^{-1}_{n+1}\cdots k'^{-1}_{i+1}v'_i+u^{-1}k'^{-1}_{n+1}\cdots k'^{-1}_{i+2}v'_{i+1}\right)\prod_{j=i+1}^n H_{\sigma(j)}\Bigl(u^{-1}\bigl(\prod_{l=j+2}^{n+1} k'_l\bigr)^{-1}v'_{j+1}\Bigr) \biggr].
\end{align}
We then have that $\tau_{A,x}(f_0,\ldots,f_i* f_{i+1},\ldots,f_{n+1})$ is the integral in \eqref{regular-integral-B}.

Now we look at the last term of \eqref{regular-cocycle}.
\begin{align}\label{start-integral-C}
    \tau_{A,x}(&f_{n+1}*f_0,\ldots,f_n) =\\
\notag    \int_{h\in M_K/Z_{M_K}(x)} & \int_{u\in K}  \int_{w\in A^\perp} \int_{G_0^n}\int_{(k,v)\in G_0} f_{n+1}(k,v)\\
 \notag   & f_0\Bigl(k^{-1}uhxh^{-1}u^{-1}k^{-1}_n\cdots k^{-1}_1, -k_1\cdots k_n uht^{-1}h^{-1}u^{-1}kv\\
 \notag   &\qquad +k_1\cdots k_nuw -\sum_{j=1}^n \left(\prod_{l=1}^j k_l\right) v_j\Bigr)\\
 \notag   & \prod_{j=1}^n f_j(k_j,v_j)\left[ \sum_{\sigma\in S_n} \sign(\sigma) H_{\sigma(j)}\left(u^{-1}\left(\prod_{l=j+1}^n k_l\right)^{-1}v_j\right)   \right].
\end{align}
Let $u'=k^{-1}u$, and $-u^{-1}kv=v'+v''$ where $v'\in A$ and $v''\in A^\perp$. By Lemma~\ref{lemma-orthogonal-complement-is-invariant}, $A^\perp$ is $M_K$-invariant. We then obtain
\begin{align*}
    & \biggl(k^{-1}uhxh^{-1}u^{-1}k_n^{-1}\cdots k_1^{-1},-k_1\cdots k_nuhx^{-1}h^{-1}u^{-1}kv\\
    & \qquad +k_1\cdots k_n uw - \sum_{j=1}^n\left(\prod_{l=1}^j k_l\right)v_j\biggr)\\
    = &  \biggl( u'hxh^{-1}u'^{-1}k^{-1}k_n^{-1}\cdots k_1^{-1},k_1\cdots k_nku'hx^{-1}h^{-1}(v'+v'')\\
    &\qquad +k_1\cdots k_nku'w- k_1\cdots k_n k v - k_1\cdots k_n k u'(v'+v'')\\
    & \qquad -\sum_{j=1}^n\left(\prod_{l=1}^j k_l\right)v_j\biggr)\\
    = & \biggl( u'hxh^{-1}u'^{-1}k^{-1}k_n^{-1}\cdots k_1^{-1},k_1\cdots k_nku'(v'+hx^{-1}h^{-1}v''+w-v'-v'')\\
    &\qquad -k_1\cdots k_n k v -\sum_{j=1}^n\left(\prod_{l=1}^j k_l\right)v_j\biggr)\\
     = & \biggl( u'hxh^{-1}u'^{-1}k^{-1}k_n^{-1}\cdots k_1^{-1},k_1\cdots k_nku'(hx^{-1}h^{-1}v''+w-v'')\\
    &\qquad -k_1\cdots k_n k v -\sum_{j=1}^n\left(\prod_{l=1}^j k_l\right)v_j\biggr)
\end{align*}
We note that $hx^{-1}h^{-1}v''\in A^\perp$, so we shall set $w'=hx^{-1}h^{-1}v''+w-v''$. Furthermore, let $(k_{n+1},v_{n+1})=(k,v)$ so that the integral \eqref{start-integral-C} becomes
\begin{align}\label{regular-integral-C}
    \tau_{A,x}(&f_{n+1}*f_0,\ldots,f_n) =\\
\notag  &  \int_{h\in M_K/Z_{M_K}(x)}  \int_{u'\in K}  \int_{w'\in A^\perp} \int_{G_0^{n+1}} \\
 \notag   & f_0\Bigl(u'hxh^{-1}u'^{-1}k_{n+1}^{-1}\cdots k_1^{-1},k_1\cdots k_{n+1}u'w' -\sum_{j=1}^{n+1}\left(\prod_{l=1}^j k_l\right)v_j\Bigr)\\
         \notag   & \prod_{j=1}^{n+1} f_j(k_j,v_j)\left[ \sum_{\sigma\in S_n} \sign(\sigma) H_{\sigma(j)}\left(u'^{-1}\left(\prod_{l=j+1}^{n+1} k_l\right)^{-1}v_j\right)   \right].
\end{align}

Now we substitute \eqref{regular-integral-A}, \eqref{regular-integral-B}, and \eqref{regular-integral-C} into \eqref{regular-cocycle}. To finish the proof for the cocycle property, we need only to show that
\begin{align}
    \label{regular-H-function} & \sum_{\sigma\in S_n}\sign(\sigma) \prod_{j=1}^n H_{\sigma(j)}\left(u^{-1}\left(\prod_{l=j+2}^{n+1} k_l\right)^{-1}v_{j+1}\right)\\
   \notag +&\sum_{i=1}^n\Biggl\{ (-1)^i \sum_{\sigma\in S_n}\sign(\sigma) \prod_{j=1}^{i-1} H_{\sigma(j)}\left(u^{-1}\left(\prod_{l=j+1}^{n+1} k_l\right)^{-1}v_j\right)\\
   \notag\cdot & H_{\sigma(i)}\left(u^{-1}k^{-1}_{n+1}\cdots k^{-1}_{i+1}v_i+u^{-1}k^{-1}_{n+1}\cdots k^{-1}_{i+2}v_{i+1}\right) \prod_{j=i+1}^n H_{\sigma(j)}\left(u^{-1}\left(\prod_{l=j+2}^{n+1}k_l\right)^{-1}v_{j+1}\right)\Biggr\}\\
   \notag  +&(-1)^{n+1}\sum_{\sigma\in S_n}\sign(\sigma) \prod_{j=1}^n H_{\sigma(j)}\left(u^{-1}\left(\prod_{l=j+1}^{n+1} k_l\right)^{-1}v_j\right)
\end{align}
is zero. Then the expression \eqref{regular-H-function} is
\begin{align*}
 \sum_{\sigma\in S_n} &\sign(\sigma) \Biggl\{\prod_{j=1}^n H_{\sigma(j)}\left(u^{-1}\left(\prod_{l=j+2}^{n+1}k_l\right)^{-1}v_{j+1}\right)\\
 &+\sum_{i=1}^n(-1)^i\biggl[\prod_{j=1}^iH_{\sigma(j)}\left(u^{-1}\left(\prod_{l=j+1}^{n+1} k_l\right)^{-1}v_j\right)\prod_{j=i+1}^nH_{\sigma(j)}\left(u^{-1}\left(\prod_{l=j+2}^{n+1}k_l\right)^{-1}v_{j+1}\right)\\
 &\qquad + \prod_{j=1}^{i-1}H_{\sigma(j)}\left(u^{-1}\left(\prod_{l=j+1}^{n+1} k_l\right)^{-1}v_j\right)\prod_{j=i}^nH_{\sigma(j)}\left(u^{-1}\left(\prod_{l=j+2}^{n+1}k_l\right)^{-1}v_{j+1}\right)\biggr]\\
 &+(-1)^{n+1}\prod_{j=1}^n H_{\sigma(j)}\left(u^{-1}\left(\prod_{l=j+1}^{n+1} k_l\right)^{-1}v_j\right)\Biggr\}.
\end{align*}
It is clear that all the terms of the sum within the curly brackets all cancel out. Thus we get that the above expression is zero and we have proven that $\tau_{A,x}$ is a cocycle.

\qed

\subsection{Cyclic condition}

Now we show that $\tau_{A,x}$ is cyclic. In other words, we need to show that
$$
\tau_{A,x}(f_n,f_0,\ldots,f_{n-1})=(-1)^n\tau_{A,x}(f_0,f_1,\ldots,f_n).$$
The left hand side is equal to
\begin{align*}
    \int_{h\in M_K/Z_{M_K}(x)}\int_{u\in K}&\int_{w\in A^\perp}\int_{G_0^n} f_n\left(uhxh^{-1}u^{-1}k^{-1}_n\cdots k^{-1}_1,k_1\cdots k_nuw-\sum_{j=1}^n\bigl(\prod_{l=1}^j k_l\bigr)v_j\right)\\
    & \prod_{j=0}^{n-1} f_j(k_{j+1},v_{j+1})\left[\sum_{\sigma\in S_n}\sign(\sigma) \prod_{j=1}^nH_{\sigma(j)}\left(u^{-1}\left(\prod_{l=j+1}^nk_l\right)^{-1}v_j\right)   \right].
\end{align*}
We make the following change of variables
\begin{align*}
    & k'_{j-1}=k_j\qquad(j=2,\ldots,n)\\
    & k'_n=uhxh^{-1}u^{-1}k^{-1}_n\cdots k^{-1}_1 \qquad \Rightarrow \qquad k_1=k'^{-1}_nuhxh^{-1}u^{-1}k'^{-1}_{n-1}\cdots k'^{-1}_1.
\end{align*}
Furthermore, let $u'=k_n'^{-1}u$ to obtain
$$
k_1=u'hxh^{-1}u'^{-1}k_n'^{-1}\cdots k_1'^{-1}.$$
Then we have
\begin{align*}
    \int_{h\in M_K/Z_{M_K}(x)}\int_{u'\in K} & \int_{w\in A^\perp} \int_{G_0^n} f_n\left(k'_n,u'hxh^{-1}w-\sum_{j=1}^n\left(u'hxh^{-1}u'^{-1}\Bigl(\prod_{l=j}^n k'_l\Bigr)^{-1}\right)v_j   \right)\\
    & f_0(u'hxh^{-1}u'^{-1}k'^{-1}_n\cdots k'^{-1}_1,v_1)\prod_{j=1}^{n-1}f_j(k'_j,v_{j+1})\\
    &\left[\sum_{\sigma\in S_n}\sign(\sigma) \prod_{j=1}^n H_{\sigma(j)}\left(u'^{-1}\left(\prod_{l=j}^n k'_l\right)^{-1}v_j\right)      \right].
\end{align*}
We make the next change of variables
\begin{align*}
    & v_{j-1}'= v_j \qquad (j=2,\ldots,n)\\
    & v_n'=u'hxh^{-1}w-\sum_{j=1}^n u'hxh^{-1}u'^{-1}\left(\prod_{l=j}^n k_l'\right)^{-1} v_j \\
    \Rightarrow & v_1=k_1'\cdots k_n'u'w-k_1'\cdots k_n' u'hx^{-1}h^{-1}u'^{-1}v'_n-\sum_{j=2}^n \prod_{l=1}^{j-1} k_l' v_{j-1}'
\end{align*}
and we obtain
\begin{align*}
    \int_{h\in M_K/Z_{M_K}(x)}\int_{u'\in K} & \int_{w\in A^\perp} \int_{G_0^n} 
     f_0\biggl( u'hxh^{-1}u'^{-1}k'^{-1}_n\cdots k'^{-1}_1,k_1'\cdots k_n'u'w\\
     &-k_1'\cdots k_n' u'hx^{-1}h^{-1}u'^{-1}v_n'-\sum_{j=1}^{n-1}\prod_{l=1}^j k_l' v_j'\biggr)\prod_{j=1}^{n}f_j(k'_j,v_j')\\
    &\Biggl[\sum_{\sigma\in S_n}\sign(\sigma)H_{\sigma(1)}\left(w-hx^{-1}h^{-1}u'^{-1}v_n'-u'^{-1}\sum_{j=1}^{n-1}\Bigl(\prod_{l=j+1}^n k_l'\Bigr)^{-1} v_j'\right)\\ 
    &\prod_{j=1}^n H_{\sigma(j)}\left(u'^{-1}\left(\prod_{l=j}^n k'_l\right)^{-1}v_{j-1}\right)      \Biggr].
\end{align*}
Now set $hx^{-1}h^{-1}u'^{-1}v_n'=w_1+w_2$ where $w_1\in A$ and $w_2\in A^\perp$. Note that we get the following equality from the expression inside $f_0$
\begin{align*}
    & k_1'\cdots k_n'u'w-k_1'\cdots k_n' u'hx^{-1}h^{-1}u'^{-1}v_n'-\sum_{j=1}^{n-1}\prod_{l=1}^j k_l' v_j'\\
    = & k_1'\cdots k_n'u'w-k_1'\cdots k_n' u'(w_1+w_2)\\
    &-k_1'\cdots k_n' v_n' + k_1'\cdots k_n'u'hxh^{-1}(w_1+w_2)-\sum_{j=1}^{n-1}\prod_{l=1}^j k_l' v_j'\\
    = & k_1'\cdots k_n'u'(w-w_1-w_2+hxh^{-1}w_1+hxh^{-1}w_2)-\sum_{j=1}^n\prod_{l=1}^j k_l' v_j'\\
    = & k_1'\cdots k_n'u'(w-w_2+hxh^{-1}w_2)-\sum_{j=1}^n\prod_{l=1}^j k_l' v_j'.\\
\end{align*}
The last equality is because $h$ and $x$ are elements of $M$, thus are stabilizers of $A$ and we have $w_1\in A$. Let $w'=w-w_2+hxh^{-1}w_2$ be the change of variables in $A^\perp$. Furthermore, we make the following observations about the $H$-function. In particular, by taking the same decomposition $hx^{-1}h^{-1}u'^{-1}v_n'=w_1+w_2$, we see that
\begin{align*}
    & H_{\sigma(1)}\left(w-hx^{-1}h^{-1}u'^{-1}v_n'-u'^{-1} \sum_{j=1}^{n-1}\Bigl(\prod_{l=j+1}^n k_l'\Bigr)^{-1}v_j'\right) \\
    = & H_{\sigma(1)}\left(w-w_1-w_2-u'^{-1} \sum_{j=1}^{n-1}\Bigl(\prod_{l=j+1}^n k_l'\Bigr)^{-1}v_j'\right) \\
    = & H_{\sigma(1)}\left(-w_1-u'^{-1} \sum_{j=1}^{n-1}\Bigl(\prod_{l=j+1}^n k_l'\Bigr)^{-1}v_j'\right). \\
\end{align*}
We also note that 
$$
H_{\sigma(1)}(u'^{-1}v_n')=H_{\sigma(1)}(hxh^{-1}w_1+hxh^{-1}w_2)=H_{\sigma(1)}(w_1),$$
and thus
$$
H_{\sigma(1)}\left(-w_1-u'^{-1} \sum_{j=1}^{n-1}\Bigl(\prod_{l=j+1}^n k_l'\Bigr)^{-1}v_j'\right)=H_{\sigma(1)}\left(-u'^{-1}\sum_{j=1}^n \Bigl(\prod_{l=j+1}^n k_l'\Bigr)^{-1}v_j'\right).$$
Thus the integral becomes
\begin{align*}
    \int_{h\in M_K/Z_{M_K}(x)}\int_{u'\in K} & \int_{w'\in A^\perp} \int_{G_0^n} 
     f_0(u'hxh^{-1}u'^{-1}k'^{-1}_n\cdots k'^{-1}_1,k_1'\cdots k_n'u'w'-\sum_{j=1}^n\prod_{l=1}^j k_l' v_j')\\
     &\prod_{j=1}^{n}f_j(k'_j,v_j')\Biggl[\sum_{\sigma\in S_n}\sign(\sigma)H_{\sigma(1)}\left(-u'^{-1}\sum_{j=1}^n\Bigl(\prod_{l=j+1}^n k_l'\Bigr)^{-1} v_j'\right)\\ 
    &\prod_{j=1}^n H_{\sigma(j)}\left(u'^{-1}\left(\prod_{l=j}^n k'_l\right)^{-1}v_{j-1}\right)      \Biggr].
\end{align*}

The expression inside the square bracket is equal to 
\begin{align*}
    -\sum_{\sigma\in S_n}\sign(\sigma) \sum_{j=1}^{n}H_{\sigma(1)}\left(u'^{-1}\left(\prod_{l=j+1}^n k'_l\right)^{-1}v'_j\right)\prod_{p=1}^{n-1} H_{\sigma(p+1)}\left(u'^{-1}\left(\prod_{l=j+1}^n k'_l\right)^{-1}v'_j\right).
\end{align*}
Note that for a single term where $j<n$, the parameter
$$
u^{-1}\left(\prod_{l=j+1}^n k'_l\right)^{-1}$$
appears exactly twice inside the $H$-functions. In particular, they appear in $H_{\sigma(1)}$ and in $H_{\sigma(j+1)}$. There is precisely one $\sigma'\in S_n$ that swaps \emph{only} $\sigma(1)$ and $\sigma(j+1)$ in the index, and thus $\sign(\sigma')=-\sign(\sigma)$. Thus the corresponding term for a fixed $j<n$, $\sigma$ and $\sigma'$ will cancel. We are left with
$$
-\sum_{\sigma\in S_n}\sign(\sigma)H_{\sigma(1)}\left(u'^{-1}\left(\prod_{l=j+1}^n k'_l\right)^{-1}v'_n\right)\prod_{p=1}^{n-1}H_{\sigma(p+1)}\left(u'^{-1}\left(\prod_{l=j+1}^nk'_l\right)^{-1}v'_j\right).
$$
Apply the permutation $(\sigma(2),\ldots,\sigma(n),\sigma(1))\mapsto (\sigma(1),\ldots,\sigma(n))$ to the above sum. Note that the sign of this permutation is $(-1)^{n-1}$. We thus obtain
$$
(-1)^n \sum_{\sigma\in S_n} \sign(\sigma)\prod_{p=1}^n H_{\sigma(p)}\left(u'^{-1}\left(\prod_{l=j+1}^nk'_l\right)^{-1}v'_j\right).
$$
We have thus proven the cyclic property.

\qed

\section{Deformation of higher orbital integrals and limit formula}\label{section-deformation-higher-orbital-integral}

Let $\GG$ be the deformation to the normal cone of the reductive group $G$, as constructed in Subsection~\ref{subsection-deformation}. Let $G_t$ be the fiber over $t\in \R$. Let $S_t=M_tA_tN_t$ be  a cuspidal parabolic subgroup of $G_t$. We have that $M_t$ is a reductive group and, if $K_M$ is a maximal compact subgroup of $M_t$ (note that $K_M$ does not depend on $t$), then $\rank K_M = \rank M_t$.

Recall that for an element $g\in G$, we can write
$$
g=\kappa(g)\mu(g)e^{H(g)}\mathbf{n}(g)$$
(see Subsection~\ref{subsection-review-orbital-integral-reductive}). Suppose we have an element $g_t\in G_t$ written as
$$
g_t=km_t\exp(tX)\exp(tY)$$
where $k\in K$, $m_t\in M_t$, $X\in \mathfrak{a}$, and $Y\in \mathfrak{n}$. Furthermore, let 
$$
m_t=k_M\exp(tX_M)\exp(tY_M),$$
under the Iwasawa decomposition of $M$. Under the deformation to the normal cone, we have 
$$
(km_t\exp(tX)\exp(tY),t)\rightarrow (kk_M, X_M+X+\phi(Y_M+Y),0)$$
as $t\rightarrow 0$. 

Note, however, that $H(g_t)=tX$, whose limit is zero as $t$ goes to zero. So we will modify this function.\footnote{Compare this with \cite[Lemma~3.13]{AfgoustidisContraction}.} Let $\{H_1,\ldots,H_n\}$ be an orthogonal basis for $\mathfrak{a}$. 
\begin{definition}
    For $t\neq 0$, define the map $H_t:G_t\rightarrow \mathfrak{a}$ by
    $$
    H_t(g)=t^{-1}H(g).$$
    Following \cite{ST2019higher}, we can write $H_t=(H_{1,t},\ldots,H_{n,t})$ where $H_{j,t}:G_t\rightarrow \R$ is defined by
    $$
    H_{j,t}(g)=\langle H_t(g),H_j\rangle.$$
\end{definition}
Note that if $g=km_t\exp(tX)\exp(tY)$ under the $KMAN$ decomposition, then $H_t(g)=X$.

In Definition~\ref{definition-H-components}, we defined maps $H_j:V\rightarrow \R$ for a general motion group $K\ltimes V$ where $V$ was an arbitrary Euclidean space. For the Cartan motion group $G_0=K\ltimes \mathfrak{p}$, we shall denote these maps as
$$
H_{j,0}:\mathfrak{p}\rightarrow \R$$
defined by
$$
H_{j,0}(X)=\langle X, H_j\rangle.$$

\begin{lemma}\label{H-function-limit-lemma}
    For any $Z\in \mathfrak{p}$, the following limit holds
    \begin{equation*}
        \lim_{t\rightarrow 0} H_{j,t}(k\exp(tZ)) = H_{j,0}(Z).
    \end{equation*} 
\end{lemma}

\begin{proof}
    For each $t\neq 0$, let $k\exp(tZ)=k'm_t\exp(tX)\exp(tY)$, where $k'\in K$, $m_t\in M_t$, $X\in \mathfrak{a}$, and $Y\in \mathfrak{n}$. Furthermore, let $m_t=k_M\exp(tX_M)\exp(tY_M)$ under the Iwasawa decomposition. As $t\rightarrow 0$, we obtain
    \begin{align*}
        (k\exp(tZ),t) & \rightarrow (k,Z,0)\\
        (k'm_t\exp(tX)\exp(tY),t)&\rightarrow (kk_M,X_M+X+\phi(Y_M+Y),0).
    \end{align*}
    Since the topology of the deformation to the normal cone is Hausdorff, we must have
    $$
    Z=X_M+X+\phi(Y_M+Y).$$
    Note that the image of $\phi$ is in the orthogonal complement of $\mathfrak{a}$, and thus we have
    \begin{align*}
        H_{j,t}(km_t\exp(tX)\exp(tY))&=\langle X,H_j\rangle\\
        H_{j,0}(Z)& = \langle X,H_j\rangle.
    \end{align*}
    Now it is clear that the limit holds. 
\end{proof}

\subsection{Cyclic Cocycle on a regular element in $M$}

For $t\neq 0$, let $P_t=M_tA_tN_t$ be a cuspidal parabolic subgroup of $G_t$, together with $\mathfrak{a}$ and $\mathfrak{n}$ as the Lie algebras of $A_t$ and $N_t$, respectively (the Lie algebras do not depend on $t$). Recall that $M_tA_t=Z_{G_t}(\mathfrak{a})$. Let $\mathfrak{m}$ be the Lie algebra of $M$ and write $\mathfrak{m}=\mathfrak{k}_M\oplus \mathfrak{p}_M$ as its Cartan decomposition. We also consider the Iwasawa decomposition $M_t=K_MA_{M,t}N_{M,t}$, and let $\mathfrak{a}_M$ and $\mathfrak{n}_M$ be the Lie algebras of $A_{M,t}$ and $N_{M,t}$, respectively. It is important to note that $K_M=K\cap M$. See \cite[Subsection~VII.7]{KnappBeyond} for details.
\begin{lemma}\label{Lie-algebra-complement}
    Let $\mathfrak{a}^\perp$ be the orthogonal complement of $\mathfrak{a}$ in $\mathfrak{p}$. Then
    $$
    \mathfrak{a}^\perp=\mathfrak{p}_M\oplus \phi(\mathfrak{n}).$$
\end{lemma}

\begin{proof} Let $\mathfrak{m}_s\oplus\mathfrak{a}_s\oplus\mathfrak{n}_s$ be a minimal parabolic subalgebra of $\mathfrak{g}$.
    Recall that $\phi(\mathfrak{n}_s)=\mathfrak{a}_s^\perp$, the orthogonal complement of $\mathfrak{a}_s$ in $\mathfrak{p}$. Furthermore, we have $\phi(\mathfrak{n}_M)=\mathfrak{a}_M^\perp$, the orthogonal complement of $\mathfrak{a}_M$ in $\mathfrak{p}_M$. Now we have
    \begin{align*}
        \mathfrak{p}\cong & \mathfrak{a}_s\oplus\mathfrak{n}_s\\
        \mathfrak{a}_s = & \mathfrak{a}_M\oplus \mathfrak{a}\\
        \mathfrak{n}_s=&\mathfrak{n}_M\oplus\mathfrak{n}.
    \end{align*}
    This implies that 
    $$
    \mathfrak{p}\cong \mathfrak{a}_M\oplus\mathfrak{a}\oplus\mathfrak{n}_M\oplus\mathfrak{n}.$$
    By dimensional consideration, we have
    \begin{equation}\label{decomposition}\mathfrak{a}^\perp \cong \mathfrak{a}_M\oplus\mathfrak{n}_M\oplus\mathfrak{n}.\end{equation}
    Note that $\mathfrak{p}_M\cong \mathfrak{a}_M\oplus\mathfrak{n}_M.$ Apply the map $\phi$ to the right hand side of \eqref{decomposition} to obtain
    $$
    \mathfrak{a}^\perp =\mathfrak{p}_M\oplus\phi(\mathfrak{n}).$$
\end{proof}

Let $T$ be a compact Cartan subgroup of $M_t$. Let $x\in M_t$ be a regular element in $M_t$. We notice \cite[Theorem~7.108]{KnappBeyond} that $x$ is in just one Cartan subgroup of $M_t$. Thus, from now on, we can assume the regular\footnote{For further details about regular elements, see \cite[Section~VII.8]{KnappBeyond}.} element $x$ is in $T$. Furthermore, we assume that $Z_{M_t}(x)=T$, and that $\operatorname{Vol}(T)=1$. 

\begin{remark}\label{remark-on-regular-elements}
    It should be noted that a regular element in $M_t$ is not necessarily a regular element in $G_t$. Furthermore, in general, it is not always the case that if $x$ is a regular element in a general reductive group, then $Z_G(x)$ is a Cartan subgroup. See for example the discussion following Equation~(7.95) in \cite[Section~VII.8]{KnappBeyond}.
\end{remark}

Let $f_0,f_1,\ldots,f_n\in C^\infty_c(\GG)$ and define $f_{j,t}\in C^\infty_c(G_t)$ as 
\begin{align*}
    f_{j,t}(g) = & f_j(g,t) & t\neq 0\\
    f_{j,0}(k,X) = & f_j(k,X,0) & t=0.
\end{align*}
Let $x\in T$ be a regular element. We define the higher orbital integrals for the fibers $G_t$, $t\neq 0$, over the regular element $x$ as
\begin{align}
    \label{higher-orbital-regular} \Phi_{P_t,x,t}&(f_{0,t},f_{1,t},\ldots,f_{n,t}) = \int_{h\in M_t/Z_{M_t}(x)} \int_{KN_t} \int_{G_t^n} f_{0,t} (khxh^{-1}nk^{-1}g_n^{-1}\cdots g_1^{-1})\\
   \notag &f_{1,t}(g_1)\cdots f_{n,t}(g_n) \Biggl[\sum_{\sigma\in S_n} \sign(\sigma) H_{\sigma(1),t}(g_1\cdots g_n k)\cdots H_{\sigma(n),t}(g_n k)\Biggr]\\
   \notag & \qquad d_tg_1\cdots d_tg_n dkd_tn\;d_t[h].
\end{align}

\begin{theorem}\label{thm:limit}
    Let $x\in T$ be a regular element in $M_t$ for all $t\neq 0$. The following limit holds .
    $$
    \lim_{t\rightarrow 0} \Phi_{P_t,x,t}(f_{0,t},f_{1,t},\ldots,f_{n,t})=\det\nolimits_{\mathfrak{p}_M}(\Ad_x-\Id)^{-1} J_\phi^{-1}\tau_{\mathfrak{a},x}(f_{0,0},f_{1,0},\ldots,f_{n,0})
    $$
\end{theorem}

\begin{proof}
    By Cartan decomposition and change of variables, we obtain 
    \begin{align*}
        \Phi_{P_t,x,t}(&f_{0,t},f_{1,t},\ldots,f_{n,t})=\int_{h\in M_t} \int_{k\in K} \int_{Y\in \mathfrak{n}} \int_{K^n}\int_{\mathfrak{p}^n} \\
        & f_{0,t}\Bigl(khxh^{-1}\exp(tY)k^{-1}\exp(-tX_n)k_n^{-1}\cdots \exp(-tX_1)k_1^{-1}\Bigr)\\
        & f_{1,t}\Bigl(k_1\exp(tX_1)\Bigr)\cdots f_{n,t}\Bigl(k_n\exp(tX_n)\Bigr)J(tX_1)\cdots J(tX_n)\\
        &\Biggl[\sum_{\sigma\in S_n} \sign(\sigma) H_{\sigma(1),t}\biggl(k_1\exp(tX_1)\cdots k_n\exp(tX_n)k\biggl)\cdots H_{\sigma(n),t}\biggl(k_n\exp(tX_n)k\biggr)\Biggr]\\
        & d_th\;dk\;dY\;dk_1\cdots dk_n\;dX_1\cdots dX_n
    \end{align*}
As we take the limit $t\rightarrow 0$, we make the following four observations.

\noindent\textbf{Observation 1} \begin{align*}
        & H_{\sigma(j),t}\biggl(u_j\exp(tX_j)\cdots u_n\exp(tX_n)k\biggr)\\
        = & H_{\sigma(j),t}\biggl(u_j\cdots u_nk\exp(t\Ad_{k^{-1}u_n^{-1}\cdots u_{j+1}^{-1}}X_j)\cdots \exp(t\Ad_{k^{-1}}X_n) \biggr)\\
        = & H_{\sigma(j),t}\Biggl(u_j\cdots u_nk\exp\biggl(t\big(\Ad_{k^{-1}u_n^{-1}\cdots u_{j+1}^{-1}}X_j+\cdots+\Ad_{k^{-1}}X_n+O(t)\big)\biggr)\Biggr)\\
        \rightarrow & H_{\sigma(j),0}\biggl(\Ad_{k^{-1}u_n^{-1}\cdots u_{j+1}}X_j+\cdots+\Ad_{k^{-1}}X_n\biggr)
        \end{align*}
        The last limit is due to Lemma~\ref{H-function-limit-lemma}.
    
    \noindent\textbf{Observation 2} For $j=1,\ldots,n$, we have
        $$f_{j,t}\big(u_j\exp(tX_j)\big)\rightarrow f_{j,0}(u_j,X_j).$$

\noindent\textbf{Observation 3}
\begin{align}
\notag & f_{0,t}\big(khxh^{-1}\exp(tY)k^{-1}\exp(-tX_n)k_n^{-1}\cdots \exp(-tX_1)k_1^{-1}\big)     \\
  \label{function-zero}  = & f_{0,t}\big(khxk^{-1}k_n^{-1}\cdots k_1^{-1}\exp(t\Ad_{k_1\cdots k_n k}Y)\exp(-t\Ad_{k_1\cdots k_n}X_n)\cdots \exp(-t\Ad_{k_1}X_1)\big)
\end{align}
Let 
        \begin{align*}
            Y_1 & =\frac{1}{2}(\Ad_{u_1\cdots u_nk}Y-\theta\Ad_{u_1\cdots u_nk}Y)  \in  \mathfrak{k}\\
            Y_2 & = \phi(\Ad_{u_1\cdots u_nk}Y)  \in \mathfrak{p}.
        \end{align*}
        Then by the Baker-Campbell-Hausdorff formula, the function in \eqref{function-zero} becomes 
\begin{align*}
    & f_{0,t}\Big(khxh^{-1}k^{-1}k_n^{-1}\cdots k_1^{-1}\exp\big(t(Y_1+Y_2)\big)\exp\big(-t\Ad_{k_1\cdots k_n}X_n\big)\cdots \exp(-t\Ad_{k_1}X_1)\Big)\\
    = & f_{0,t}\Big( khxh^{-1}k^{-1}k_n^{-1}\cdots k_1^{-1} \exp(tY_1)\exp\big(-tY_1\big)\exp\big(t(Y_1+Y_2)\big)\\
    & \qquad\exp\big(-t\Ad_{k_1\cdots k_n}X_n\big)\cdots \exp\big(-t\Ad_{k_1}X_1\big)\Big)\\
    = & f_{0,t}\Biggl(khxh^{-1}k^{-1}k_n^{-1}\cdots k_1^{-1}\exp(tY_1)\\
    & \qquad \exp\Big(t\big(Y_2-\Ad_{k_1\cdots k_n}X_n-\cdots -\Ad_{k_1}X_1+O(t)\big)\Big)\Biggr).
\end{align*}
Let $h\in M_t$ ($t\neq 0)$ be given as $h=u\exp(tv)$ where $u\in K_M$ and $v\in \mathfrak{p}_M$. Define
\begin{align*}
F\big(u\exp(tv),t\big)=&f_{0,t}\Biggl(k\exp(tv)k^{-1}k_n^{-1}\cdots k_1^{-1}\exp(tY_1)\\
& \quad\exp\Big(t\big(Y_2-\Ad_{k_1\cdots k_n}X_n-\cdots - \Ad_{k_1}X_1+O(t)\big)\Big)\Biggl).    
\end{align*}
Note that 
\begin{align*}
    &f_{0,t}\Biggl(ku\exp(tv)k^{-1}k_n^{-1}\cdots k_1^{-1}\exp(tY_1)\\
& \quad\exp\Big(t\big(Y_2-\Ad_{k_1\cdots k_n}X_n-\cdots - \Ad_{k_1}X_1+O(t)\big)\Big)\Biggr)\\
= & f_{0,t}\Biggl(kuk^{-1}k_n^{-1}\cdots k_1^{-1}\exp(t\Ad_{k_1\cdots k_nk}v)\exp(tY_1)\\
& \quad\exp\Big(t\big(Y_2-\Ad_{k_1\cdots k_n}X_n-\cdots - \Ad_{k_1}X_1+O(t)\big)\Big)\Biggl)\\
= & f_{0,t}\Biggl(kuk^{-1}k_n^{-1}\cdots k_1^{-1}\exp(tY_1)\exp\Big(t\big(\Ad_{k_1\cdots k_nk}v+Y_2\\
& \qquad -\Ad_{k_1\cdots k_n}X_n-\cdots -\Ad_{k_1}X_1+O(t)\big)\Big)\Biggr)\\
\rightarrow & f_{0,0}\Big(kuk^{-1}k_n^{-1}\cdots k_1^{-1},\Ad_{k_1\cdots k_nk}v+Y_2-\Ad_{k_1\cdots k_n}X_n-\cdots - \Ad_{k_1}X_1\Big),
\end{align*}
as we take the limit $t\rightarrow 0$. So we shall define
\begin{align*}
F(u,v,0)=&f_{0,0}(kuk^{-1}k_n^{-1}\cdots k_1^{-1},\Ad_{k_1\cdots k_nk}v+Y_2\\
& \qquad -\Ad_{k_1\cdots k_n}X_n-\cdots - \Ad_{k_1}X_1).
\end{align*}
By Theorem~\ref{integral-KAK-decomposition}, as $t\rightarrow 0$, we obtain
\begin{align*}
    \int_{M_t} F(hxh^{-1},t)\;d_th\rightarrow & \int_{u\in K_M}\int_{v\in \mathfrak{p}_M} F(uxu^{-1},\Ad_{uxu^{-1}}v-v,0)\;du\;dv\\
    \rightarrow & \int_{u\in K_M}\int_{v\in\mathfrak{p}_M} f_{0,0}\Big(kuxu^{-1}k_n^{-1}\cdots k_1^{-1},\Ad_{k_1\cdots k_nk}(\Ad_{uxu^{-1}}v-v)\\
    & \qquad +Y_2-\Ad_{k_1\cdots k_n}X_n-\cdots \Ad_{k_1}X_1\Big).
\end{align*}

\noindent\textbf{Observation 4}
Recall from Lemma~\ref{lemma-jacobian-limit} that $J(tX_j)\rightarrow 1$ as $t\rightarrow 0$.

Combining the four observations, as $t\rightarrow 0$, we obtain
\begin{align*}
        \Phi_{P_t,x,t}&(f_{0,t},f_{1,t},\ldots,f_{n,t})\rightarrow \int_{u\in K_M}\int_{v\in \mathfrak{p}_M}\int_{k\in K}\int_{Y\in \mathfrak{n}}\int_{K^n}\int_{\mathfrak{p}^n} \\
        &f_{0,0}\Big(kuxu^{-1}k_n^{-1}\cdots k_1^{-1},\Ad_{k_1\cdots k_nk}(\Ad_{uxu^{-1}}v-v)+\phi(\Ad_{k_1\cdots k_nk}Y)\\
        & -\Ad_{k_1\cdots k_n} X_n-\cdots -\Ad_{k_1}X_1\Big)f_{1,0}(k_1,X_1)\cdots f_{n,0}(k_n,X_n)\\
        &\Biggl[\sum_{\sigma\in S_n} \sign(\sigma) H_{\sigma(1),0}(\Ad_{k^{-1}k_n^{-1}\cdots k_2^{-1}}X_1+\cdots +\Ad_{k^{-1}} X_n)\cdots H_{\sigma(n),0}(\Ad_{k^{-1}}X_n)  \Biggr].
\end{align*}
    We make the following change of variables $$v'=\Ad_{uxu^{-1}}v-v\quad \Rightarrow \quad dv'=\det\nolimits_{\mathfrak{p}_M}(\Id-\Ad_{uxu^{-1}})dv,$$ and we also use \eqref{integral-Lie-algebra-diffeo} to obtain
    \begin{align*}
        &\int_{u\in K_M}\int_{k\in K}\int_{v'\in \mathfrak{p}_M}\int_{w'\in \phi(\mathfrak{n})}\int_{K^n}\int_{\mathfrak{p}^n} \det\nolimits_{\mathfrak{p}_M}(\Ad_{uxu^{-1}}-\Id)\\
        & f_{0,0}\Big(kuxu^{-1}k_n^{-1}\cdots k_1^{-1},\Ad_{k_1\cdots k_nk}v'+\Ad_{k_1\cdots k_n k}w'\\
        &\qquad -\Ad_{k_1\cdots k_n}X_n-\cdots -\Ad_{k_1}X_1\Big) \prod_{j=1}^n f_{j,0}(k_j,X_j)J_\phi^{-1}\\
        & \Biggl[\sum_{\sigma\in S_n} \sign(\sigma) H_{\sigma(1),0}(\Ad_{k^{-1}k_n^{-1}\cdots k_2^{-1}}X_1+\cdots +\Ad_{k^{-1}} X_n)\cdots H_{\sigma(n),0}(\Ad_{k^{-1}}X_n)  \Biggr].
    \end{align*}
Using Lemma~\ref{Lie-algebra-complement}, we obtain

\begin{align*}
        &J_\phi^{-1}\int_{u\in K_M}\int_{k\in K}\int_{w\in\mathfrak{a}^\perp}\int_{K^n}\int_{\mathfrak{p}^n} \det\nolimits_{\mathfrak{p}_M}(\Ad_{uxu^{-1}}-\Id)\\
        & f_{0,0}\Big(kuxu^{-1}k_n^{-1}\cdots k_1^{-1},\Ad_{k_1\cdots k_n k}w\\
        &\qquad -\Ad_{k_1\cdots k_n}X_n-\cdots -\Ad_{k_1}X_1\Big) \prod_{j=1}^n f_{j,0}(k_j,X_j)\\
        & \Biggl[\sum_{\sigma\in S_n} \sign(\sigma) H_{\sigma(1),0}(\Ad_{k^{-1}k_n^{-1}\cdots k_2^{-1}}X_1+\cdots +\Ad_{k^{-1}} X_n)\cdots H_{\sigma(n),0}(\Ad_{k^{-1}}X_n)  \Biggr].
    \end{align*}
To finish the proof, we make two claims.

\noindent\textbf{Claim 1} $\det\nolimits_{\mathfrak{p}_M}(\Ad_{uxu^{-1}}-\Id)=\det\nolimits_{\mathfrak{p}_M}(\Ad_x-\Id)$.\\
Indeed
\begin{align*}
    \det\nolimits_{\mathfrak{p}_M}(\Ad_{uxu^{-1}}-\Id)=&\det\nolimits_{\mathfrak{p}_M}(\Ad_{uxu^{-1}}-\Ad_{uu^{-1}})\\
    = & \det\nolimits_{\mathfrak{p}_M}(\Ad_u(\Ad_x-\Id)\Ad_{u^{-1}})\\
    = & \det\nolimits_{\mathfrak{p}_M}(\Ad_x-\Id).
\end{align*}

\noindent\textbf{Claim 2} $K_M=Z_K(\mathfrak{a})$, which is the stabilizer subgroup of $\mathfrak{a}$ in $K$. See Remark~\ref{M-K-Remark} for details. \\
Indeed, if $m\in Z_K(\mathfrak{a})$, then $m\in K$ and $m\in Z_G(\mathfrak{a})$. Since $Z_G(\mathfrak{a})=MA$ (\cite[Proposition~7.82(a)]{KnappBeyond}), this implies that $m\in K\cap(MA)=K\cap M$ (see \cite[Subsection~V.5]{KnappRepTheorySemisimpleGroups} for details). Thus, $m\in K_M$. \\
On the other hand, if $m\in K_M$, then $m\in K\cap M$. This implies that $m\in K$ and $m\in Z_G(\mathfrak{a})$. Thus, $m\in Z_K(\mathfrak{a})$. 

With the two claims, we have finished the proof.

\end{proof}

\begin{remark}
    It is also interesting to study the limit of the higher orbital integrals $\Phi_{P_t,x,t}$ when $x$ is \emph{not} regular (in $M$). We noticed that a similar argument can be used to show that when $P_t$ is minimal parabolic and $x=e$, the family of higher orbital integrals $\Phi_{P_t,e,t}$ converges to a corresponding one on the Cartan motion group. For general cuspidal parabolic subgroups and general group elements, the identification of the limit is open.
\end{remark}

\bibliographystyle{alpha}

\end{document}